\documentclass[12pt,leqno,final]{article}
\usepackage{amsmath, verbatim, color}
\usepackage{mathrsfs}
\usepackage{dsfont}
\usepackage[a4paper,margin=0.91in]{geometry}
\usepackage{amsmath, amsthm, amsfonts, amssymb, color}
\usepackage{mathrsfs}
\usepackage{color}
\usepackage{stmaryrd}
\makeatletter
\newcommand{\Spvek}[2][r]{%
  \gdef\@VORNE{1}
  \left(\hskip-\arraycolsep%
    \begin{array}{#1}\vekSp@lten{#2}\end{array}%
  \hskip-\arraycolsep\right)}

\def\vekSp@lten#1{\xvekSp@lten#1;vekL@stLine;}
\def\vekL@stLine{vekL@stLine}
\def\xvekSp@lten#1;{\def\temp{#1}%
  \ifx\temp\vekL@stLine
  \else
    \ifnum\@VORNE=1\gdef\@VORNE{0}
    \else\@arraycr\fi%
    #1%
    \expandafter\xvekSp@lten
  \fi}
\makeatother

\newtheorem{thm}{Theorem}[section]

\newtheorem{lem}[thm]{Lemma}
\newtheorem{prp}[thm]{Proposition}

\newtheorem{exa}[thm]{Example}

\theoremstyle{definition}

\newcommand{\scr}[1]{\mathscr #1}
\definecolor{wco}{rgb}{0.5,0.2,0.3}

\numberwithin{equation}{section} \theoremstyle{remark}

\newcommand{\ua}{\uparrow}

\title{{\bf   Well-Posedness for McKean-Vlasov SDEs with Distribution Dependent Stable Noises}
}
\author{
{\bf     Chang-Song Deng $^{a)}$, Xing Huang $^{b)}$,    }\\
\footnotesize{  a)School of Mathematics and Statistics, Wuhan University, Wuhan 430072, China}\\
\footnotesize{ dengcs@whu.edu.cn }\\
\footnotesize{  b)Center for Applied Mathematics, Tianjin University, Tianjin 300072, China}\\
\footnotesize{  xinghuang@tju.edu.cn}}
\begin{document}
\allowdisplaybreaks
\def\R{\mathbb R}  \def\ff{\frac} \def\ss{\sqrt} \def\B{\mathbf
B} \def\W{\mathbb W}
\def\N{\mathbb N} \def\kk{\kappa} \def\m{{\bf m}}
\def\ee{\varepsilon}\def\ddd{D^*}
\def\dd{\delta} \def\DD{\Delta} \def\vv{\varepsilon} \def\rr{\rho}
\def\<{\langle} \def\>{\rangle} \def\GG{\Gamma} \def\gg{\gamma}
  \def\nn{\nabla} \def\pp{\partial} \def\E{\mathbb E}
\def\d{\text{\rm{d}}} \def\bb{\beta} \def\aa{\alpha} \def\D{\scr D}
  \def\si{\sigma} \def\ess{\text{\rm{ess}}}
\def\beg{\begin} \def\beq{\begin{equation}}  \def\F{\scr F}
\def\Ric{\text{\rm{Ric}}} \def\Hess{\text{\rm{Hess}}}
\def\e{\text{\rm{e}}}
\def\iup{\text{\rm{i}}}
\def\ua{\underline a} \def\OO{\Omega}  \def\oo{\omega}
 \def\tt{\tilde} \def\Ric{\text{\rm{Ric}}}
\def\cut{\text{\rm{cut}}} \def\P{\mathbb P} \def\ifn{I_n(f^{\bigotimes n})}
\def\C{\scr C}      \def\aaa{\mathbf{r}}     \def\r{r}
\def\gap{\text{\rm{gap}}} \def\prr{\pi_{{\bf m},\varrho}}  \def\r{\mathbf r}
\def\Z{\mathbb Z} \def\vrr{\varrho} \def\ll{\lambda}
\def\L{\scr L}\def\Tt{\tt} \def\TT{\tt}\def\II{\mathbb I}
\def\i{{\rm in}}\def\Sect{{\rm Sect}}  \def\H{\mathbb H}
\def\M{\scr M}\def\Q{\mathbb Q} \def\texto{\text{o}} \def\LL{\Lambda}
\def\Rank{{\rm Rank}} \def\B{\scr B} \def\i{{\rm i}} \def\HR{\hat{\R}^d}
\def\to{\rightarrow}\def\l{\ell}\def\iint{\int}
\def\EE{\scr E}\def\Cut{{\rm Cut}}
\def\A{\scr A} \def\Lip{{\rm Lip}}
\def\BB{\scr B}\def\Ent{{\rm Ent}}\def\L{\scr L}
\def\R{\mathbb R}  \def\ff{\frac} \def\ss{\sqrt} \def\B{\mathbf
B}
\def\N{\mathbb N} \def\kk{\kappa} \def\m{{\bf m}}
\def\dd{\delta} \def\DD{\Delta} \def\vv{\varepsilon} \def\rr{\rho}
\def\<{\langle} \def\>{\rangle} \def\GG{\Gamma} \def\gg{\gamma}
  \def\nn{\nabla} \def\pp{\partial} \def\E{\mathbb E}
\def\d{\text{\rm{d}}} \def\bb{\beta} \def\aa{\alpha} \def\D{\scr D}
  \def\si{\sigma} \def\ess{\text{\rm{ess}}}
\def\beg{\begin} \def\beq{\begin{equation}}  \def\F{\scr F}
\def\Ric{\text{\rm{Ric}}} \def\Hess{\text{\rm{Hess}}}
\def\ua{\underline a} \def\OO{\Omega}  \def\oo{\omega}
 \def\tt{\tilde} \def\Ric{\text{\rm{Ric}}}
\def\cut{\text{\rm{cut}}} \def\P{\mathbb P} \def\ifn{I_n(f^{\bigotimes n})}
\def\C{\scr C}      \def\aaa{\mathbf{r}}     \def\r{r}
\def\gap{\text{\rm{gap}}} \def\prr{\pi_{{\bf m},\varrho}}  \def\r{\mathbf r}
\def\Z{\mathbb Z} \def\vrr{\varrho} \def\ll{\lambda}
\def\L{\scr L}\def\Tt{\tt} \def\TT{\tt}\def\II{\mathbb I}
\def\i{{\rm in}}\def\Sect{{\rm Sect}}  \def\H{\mathbb H}
\def\M{\scr M}\def\Q{\mathbb Q} \def\texto{\text{o}} \def\LL{\Lambda}
\def\Rank{{\rm Rank}} \def\B{\scr B} \def\i{{\rm i}} \def\HR{\hat{\R}^d}
\def\to{\rightarrow}\def\l{\ell}
\def\8{\infty}\def\I{1}\def\U{\scr U}
\maketitle

\begin{abstract}
The well-posedness is established for McKean-Vlasov SDEs driven
by $\alpha$-stable noises ($1<\alpha<2$). In this model, the drift is H\"{o}lder continuous
in space variable and Lipschitz continuous in distribution variable with respect to the sum of
Wasserstein and weighted variation distances, while the noise coefficient satisfies
the Lipschitz condition in distribution variable with respect to the sum of two Wasserstein distances.
The main tool relies on Zvonkin's transform, a time-change technique and a two-step fixed point argument.
\end{abstract} \noindent

 \vskip 2cm

\section{Introduction}

Distribution dependent SDEs, also called McKean-Vlasov SDEs (see e.g.\ \cite{McKean}), can be used to characterize
nonlinear Fokker-Planck-Kolmogorov equations. Compared with the classical (distribution independent)
SDE, due to the dependence of coefficients (especially noise coefficients) on distribution of the solutions, the study of McKean-Vlasov SDEs
is much more difficult. One of the basic issues in the investigation of McKean-Vlasov SDEs
is the well-posedness (existence and uniqueness of solutions).

If the drift coefficients are Lipschitz continuous in distribution variable with respect to
the $L^\theta$-Wasserstein distance ($\theta\geq1$), the authors obtain in \cite{HW19} the well-posedness for
McKean-Vlasov SDEs driven by distribution dependent Brownian noises. When the drift is Lipschitz continuous
under a weighted variation distance, the well-posedness is established in \cite{RZ,W21a} only for the case that the diffusion
coefficient is distribution free. It seems that the technique (Girsanov's transform) used in \cite{RZ,W21a}
is unavailable in the distribution dependent noise case. To overcome the difficulty,
\cite{HWJMAA} adopts a parametrix method to derive the well-posedness and regularity estimates. One can refer to \cite{CF} for more
details on the parametrix method.

In recent years, McKean-Vlasov SDEs with pure jump noises have also attracted great interest. In \cite{HY}, the authors investigate the well-posedness for McKean-Vlasov SDEs with additive $\alpha$-stable noise ($1<\alpha<2$), where the drift is assumed to be $C_b^\beta$ with $\beta\in(1-\alpha/2,1)$ in space variable, and Lipschitz continuous in distribution variable with
 respect to the $L^\theta$-Wasserstein distance ($1<\theta<\alpha$). We refer the readers to \cite{JMW} for related results
 on L\'{e}vy-driven McKean-Vlasov SDEs without drift. The aim of this paper is to make some progress on the well-posedness
 for McKean-Vlasov SDEs of the following form
\begin{align}\label{E1}
\d X_t=b_t(X_t,\L_{X_t})\,\d t+\sigma_t(\L_{X_t})\,\d Z_{t},\quad t\in[0,T],
\end{align}
where $T>0$ is a fixed constant, $(Z_t)_{t \ge 0}$ is an $m$-dimensional rotationally invariant
$\alpha$-stable L\'evy process (with infinitesimal generator $-\frac12(-\triangle)^{\alpha/2}$) on a
complete filtration probability space $(\Omega,\{\scr F_t\}_{t\in[0,T]},\P)$, $\L_{X_t}$ is the law of $X_t$, and for the space
$\scr P$ of all probability measures on $\R^d$ equipped with the weak topology,
$$
    b:[0,T]\times\R^d\times\scr P\rightarrow\R^d,\quad \sigma:[0,T]\times\scr P\rightarrow\R^d\otimes\R^m
$$
are measurable.

To characterize the dependence of the coefficients on distribution variable,
we introduce some distances on $\scr P$ or its subspace.
The total variance distance $\|\cdot\|_{var}$ is given by
  $$\|\gamma-\tilde{\gamma}\|_{var} := \sup_{|f|\le 1}
  \big|\gamma(f)-\tilde{\gamma}(f)\big|,\quad\gamma,\tilde{\gamma}\in \scr P.$$
For $\kappa>0$, let
$$\scr P_\kappa=\big\{\gg\in \scr P\,;\, \gg(|\cdot|^\kappa)<\infty\big\}.$$
Recall the $L^\kappa$-Wasserstein distance $\W_\kappa$
  $$\W_\kappa(\gamma,\tilde{\gamma}):= \inf_{\pi\in \C(\gamma,\tilde{\gamma})} \left(\int_{\R^d\times\R^d} |x-y|^\kappa
  \,\pi(\d x,\d y)\right)^{1/(1\vee \kappa)
  },\quad \gamma,\tilde{\gamma}\in \scr P_\kappa,$$
  where $\C(\gamma,\tilde{\gamma})$ is the set of all couplings of $\gamma$ and $\tilde{\gamma}$.
  For $\kappa>0$, $\scr P_\kappa$ is a complete metric space under the weighted variation distance
  $$\|\gamma-\tilde{\gamma}\|_{\kappa,var} := \sup_{|f|\le 1+|\cdot|^\kappa}
  \big|\gamma(f)-\tilde{\gamma}(f)\big|,\quad\gamma,\tilde{\gamma}\in \scr P_\kappa.$$
  It is clear that $\|\gamma-\tilde{\gamma}\|_{var}\leq\|\gamma-\tilde{\gamma}\|_{\kappa,var}$ for $\kappa>0$ and
  $\gamma,\tilde{\gamma}\in \scr P_\kappa$. If $\kappa\in(0,1]$,
  the following adjoint formula (see e.g.\ \cite[Theorem 5.10]{Chen04}) holds for
  $\gamma,\tilde{\gamma}\in \scr P_\kappa$
$$\W_\kappa(\gamma,\tilde{\gamma})=\sup_{[f]_\kappa\leq 1}|\gamma(f)-\tilde{\gamma}(f)|
=\sup_{[f]_\kappa\leq 1,f(0)=0}|\gamma(f)-\tilde{\gamma}(f)|,$$
where $[f]_\kappa$ denotes the H\"{o}lder seminorm (of exponent $\kappa$) of $f:\R^d\rightarrow\R$ defined by $[f]_\kappa:=\sup_{x\neq y}\frac{|f(x)-f(y)|}{|x-y|^\kappa}$. It is easy to see that for $\gamma,\tilde{\gamma}\in \scr P_\kappa$,
$$
    \W_\kappa(\gamma,\tilde{\gamma})\leq \|\gamma-\tilde{\gamma}\|_{\kappa,var},\quad \kappa\in(0,1].
$$

  To derive the well-posedness for \eqref{E1}, we make the following assumptions.
\beg{enumerate}
\item[$(A1)$] $\alpha\in(1,2)$.
\item[$(A2)$] There exist   $\beta\in(0,1)$ satisfying $2\beta+\alpha>2$, $K_1>0$ and $k\in[1,\alpha)$ such that
$$
|b_t(x,\gamma)-b_t(y,\tilde{\gamma})|\leq K_1(\|\gamma-\tilde{\gamma}\|_{k,var}+\W_{k}(\gamma,\tilde{\gamma})+|x-y|^\beta)
$$
for all $t\in[0,T]$, $x,y\in\R^d$ and $\gamma,\tilde{\gamma}\in\scr P_k$. Moreover, 
$$\|b\|_\infty:=\sup_{t\in[0,T], x\in\R^d,\gamma\in\scr P_k}|b_t(x,\gamma)|<\infty.$$
\item[$(A3)$] There exist constants $K_2\geq1$ and $\eta\in(0,1)$ such that
$$K_2^{-1}I\leq (\sigma_t\sigma^\ast_t)(\gamma)\le K_2I,\quad\gamma\in\scr P_k $$
and
$$\|\sigma_t(\gamma)-\sigma_t(\tilde{\gamma})\|\leq K_2\big(\W_\eta(\gamma,\tilde{\gamma})+\W_k(\gamma,\tilde{\gamma})
\big),\quad
t\in[0,T],\,\gamma,\tilde{\gamma}\in\scr P_k,$$
where $k\in[1,\alpha)$ is the constant appearing in $(A2)$.
\end{enumerate}

Inspired by \cite[Example 3]{ZG}, we first present a counterexample to point out that we cannot expect the uniqueness for
the solution to \eqref{E1} if
the noise coefficient $\sigma$ is only assumed to be Lipschitz continuous
in distribution variable with respect to the total variance distance. This means that,
to guarantee the well-posedness, it is impossible to replace the Wasserstein distance $\W_\eta$
by the weighted variation distance $\|\cdot\|_{k,var}$ in $(A3)$.

\begin{exa}
    Let $d=m=1$ and $Z_t$ be an $\alpha$-stable process on $\R$ \textup{(}$1<\alpha<2$\textup{)}.
    By the asymptotic formula for the heat kernel of stable
    processes \textup{(}cf.\ \cite[Corollary 2.1\,a)]{DS19}\textup{)},
    it is not hard to verify that
    $$
        \lim_{x\rightarrow\infty}\frac{\P(Z_1\geq2x)}{\P(x\leq Z_1<2x)}=\frac{1}{2^\alpha-1}<1.
    $$
    Then we can pick large enough $M>1$ such that $\P(Z_1\geq 2M)<\P(M\leq Z_1<2M)$. Set
    \begin{align*}
        a&:=\frac{\P(M\leq Z_1<2M)-\P(Z_1\geq 2M)}{\P(M\leq Z_1<2M)}\,\in\,(0,1),\\
        b&:=\frac{1}{\P(M\leq Z_1<2M)}\,\in\,(1,\infty).
    \end{align*}
    Let
    $$
        \sigma_t(\gamma):=a+b\gamma\left([2Mt^{1/\alpha},\infty)\right),
        \quad t\in[0,T],\,\gamma\in\scr P.
    $$
    It is easy to see that for all $t\in[0,T]$ and $\gamma,\tilde{\gamma}\in\scr P$,
    $$
        0<a\leq \sigma_t(\gamma)\leq a+b,\quad\text{and}\quad
        |\sigma_t(\gamma)-\sigma_t(\tilde{\gamma})|\leq b\|\gamma-\tilde{\gamma}\|_{var}.
    $$
    Consider the  McKean-Vlasov SDE \textup{(}without drift\textup{)} on $\R$:
    \begin{equation}\label{e22hs}
        \d X_t=\sigma_t(\L_{X_t})\,\d Z_{t},\quad t\in[0,T].
    \end{equation}
    Since the distribution of $Z_t$ coincides with that of $t^{1/\alpha}Z_1$, we have
    \begin{align*}
        \sigma_t(\L_{Z_t})&=a+b\P(Z_t\geq 2Mt^{1/\alpha})=a+b\P(Z_1\geq2M)=1,\\
        \sigma_t(\L_{2Z_t})&=a+b\P(2Z_t\geq 2Mt^{1/\alpha})=a+b\P(Z_1\geq M)=2.
    \end{align*}
    This implies that the SDE \eqref{e22hs} with initial value $X_0=0$ has at least two strong
    solutions: $Z_t$ and $2Z_t$.
\end{exa}

Denote by $C([0,T];\scr P_k)$ the set of all
continuous maps from $[0,T]$ to $\scr P_k$ under the metric $\W_k$.
Throughout the paper the constant $C$
denotes positive constant which may depend on $T,d,m,\alpha,\beta,k,K_1,K_2,\|b\|_\infty$; its value may change,
without further notice, from line to line.

Our main result is the following theorem:
\begin{thm}\label{EUS} Assume $(A1)$-$(A3)$.
Then \eqref{E1} is well-posed in $\scr P_{k}$, and the solution satisfies $\L_{X_\cdot}\in C([0,T];\scr P_k)$ and
$$
    \E\left[\sup_{t\in[0,T]}|X_t|^k\right]<C\left(
    1+\E\big[|X_0|^k\big]\right).
$$
\end{thm}

The remainder of the paper is organized as follows: In Section 2, we use a fixed point argument to establish the well-posedness for McKean-Vlasov SDEs with distribution free drifts, which will be used in the proof of Theorem \ref{EUS}. The proof of our main result is
presented in Section 3. Finally,
we give in the Appendix two limits (concerning stable subordinators),
which have been used in Section 2 and might be interesting on their own.

\section{Well-posedness for SDEs with distribution free drifts}

This section is devoted to the well-posedness for a particular class of McKean-Vlasov SDEs, where the drift
coefficient does not depend on distribution variable.
For  $\mu\in C([0,T];\scr P_k)$  and $\gg\in \scr P_k$, consider the following McKean-Vlasov SDE  with initial distribution  $\L_{X_{0}^{\gg,\mu}}=\gg$:
 \beq\label{ED}
     \d X_{t}^{\gg,\mu}= b_t(X_{t}^{\gg,\mu}, \mu_t)\,\d t+\sigma_t(\L_{X_{t}^{\gg,\mu}})\,\d Z_{t},\quad t\in [0,T].
 \end{equation}

\begin{prp}\label{PW}
Assume $(A1)$-$(A3)$.
 For any $\mu\in C([0,T]; \scr P_k)$ and $\gg\in \scr P_k$,  $\eqref{ED}$ is well-posed in $\scr P_k$, and  the unique solution satisfies $\L_{X_{\cdot}^{\gg,\mu}} \in C([0,T]; \scr P_k)$. Furthermore, for any $\mu^1,\mu^2\in C([0,T];\scr P_k)$, $\gg\in \scr P_k$ and
 $\delta>0$ large enough,
\begin{align*}
    \sup_{t\in[0,T] }\e^{-\delta t}&\big[\W_\eta(\L_{X_{t}^{\gg,\mu^1}},\L_{X_{t}^{\gg,\mu^2}})+ \W_k(\L_{X_{t}^{\gg,\mu^1}},\L_{X_{t}^{\gg,\mu^2}})\big]\\
    &\le C\gamma(1+|\cdot|^k)[\delta^{1/\alpha-1}+\delta^{-1}]
    \sup_{t\in[0,T]}\e^{-\delta t}\left[
    \|\mu^1_t-\mu^2_t\|_{k,var}+\W_{k}(\mu^1_t,\mu^2_t)\right].
\end{align*}
\end{prp}

\subsection{A useful estimate for classical SDEs}

Let $\gamma\in\scr P_k$, $X_{0}^\gg$ be $\F_0$-measurable with $\L_{X_{0}^\gg}=\gg$, and
  $\mu,\nu\in C([0,T];\scr P_k)$. Consider the following (distribution independent) SDE
\beq\label{ED00}
    \d X_{s,t}^{\gg,\mu,\nu}= b_t(X_{s,t}^{\gg,\mu,\nu},\mu_t)\,\d t+\si_t(\nu_t)\,\d Z_{t},
    \quad   0\leq s\leq t\leq T
\end{equation}
with $X_{s,s}^{\gg,\mu,\nu}=X_{0}^\gg$. According to \cite{P}, \eqref{ED00} is well-posed
under the assumptions $(A2)$ and $(A3)$.
For simplicity, we denote $X_{t}^{\gg,\mu,\nu}=X_{0,t}^{\gg,\mu,\nu}$.
Moreover, if $\gg=\delta_x$ is the Dirac measure concentrated at $x\in\R^d$, we write
$X_{s,t}^{x,\mu,\nu}=X_{s,t}^{\delta_x,\mu,\nu}$ and $X_{t}^{x,\mu,\nu}=X_{0,t}^{\delta_x,\mu,\nu}$
for $0\leq s\leq t\leq T$.

\begin{lem}\label{PW0}
Assume $(A2)$ and $(A3)$. Let $\gamma\in\scr P_k$ and $\mu^i,\nu^i\in C([0,T];\scr P_k)$, $i=1,2$. Then for any $\delta>0$,
        \begin{align*}
              \sup_{t\in[0,T] }\e^{-\delta t}\W_k(\L_{X_{t}^{\gg,\mu^1,\nu^1}},\L_{X_{t}^{\gg,\mu^2,\nu^2}})
              &\leq \frac{C}{\delta}\sup_{t\in[0,T]}\e^{-\delta t}[\|\mu^1_t-\mu^2_t\|_{k,var}+\W_{k}(\mu^1_t,\mu^2_t)]\\
              &\quad+\frac{C}{\sqrt{\delta}}\sup_{t\in[0,T]}\e^{-\delta t}[\W_\eta(\nu^1_t,\nu^2_t)+\W_k(\nu^1_t,\nu^2_t)].
         \end{align*}
\end{lem}

\begin{proof} For  $\ll>0$ and $\mu,\nu\in C([0,T];\scr P_k)$, consider the following PDE  for
$u^{\ll,\mu,\nu}:[0,T]\times\R^d\to\R^d$:
\begin{equation}\label{A10}
\partial_tu^{\ll,\mu,\nu}_t(\cdot)+\A_t^{\nu} u^{\ll,\mu,\nu}_t(\cdot)+\nabla u^{\ll,\mu,\nu}_t(\cdot)
b_t(\cdot,\mu_t)+
b_t(\cdot,\mu_t)=\ll u^{\ll,\mu,\nu}_t(\cdot),\quad u^{\ll,\mu,\nu}_T(\cdot)=0,
\end{equation}
where
\beg{equation}\label{L34}
\A^{\nu}_t f(\cdot):=\int_{\mathbb{R}^{m}\backslash\{0\}}\big[f(\cdot+\sigma_t(\nu_t)y)-f(\cdot)
-\langle \sigma_t(\nu_t)y,\nabla f(\cdot)\rangle\mathds{1}_{\{ |y|\leq 1\}} \big]\,\Pi(\d y),
\end{equation}
and
$$
    \Pi(\d y):=\frac{\alpha\Gamma(\frac{m+\alpha}{2})}
    {2^{2-\alpha}\pi^{m/2}\Gamma(1-\frac\alpha2)}\,\frac{\d y}{|y|^{m+\alpha}}
$$
is the L\'{e}vy measure of $Z_t$. According to \cite[Theorem 3.4]{P}, there exists (large enough) $\lambda>0$ such that \eqref{A10} has a unique solution $u^{\ll,\mu,\nu} \in C^1\big([0,T];C_{b}^{\alpha+\beta}\left(\mathbb{R}^{d};\R^d\right)\big)$ with
\begin{equation}\label{A20}
\sup_{\mu,\nu\in C([0,T];\scr P_k),t\in[0,T]}\|\nn u_t^{\ll,\mu,\nu}(\cdot)\|_{\8}\le
\ff{1}{2},
\end{equation}
and
\beg{equation}\label{uuu}
\sup_{\mu,\nu\in C([0,T];\scr P_k),t\in[0,T]}\|u_t^{\ll,\mu,\nu}(\cdot)\|_{\infty}+\sup_{\mu,\nu\in C([0,T];\scr P_k),t\in[0,T]}\|\nabla u^{\ll,\mu,\nu}_t(\cdot)\|_{\alpha+\beta-1}<\infty.
\end{equation}
Here $C_b^{\alpha+\beta}\left(\mathbb{R}^{d};\R^d\right)\big)$ denotes the space of all bounded functions
$g:\R^d\rightarrow\R^d$ with continuous derivatives up to order $\lfloor\alpha+\beta\rfloor$ (the integer part of $\alpha+\beta$) with the norm
$$\|g\|_{\alpha+\beta}:=\sum_{i=0}^{\lfloor\alpha+\beta\rfloor}\|\nabla^i g\|_\infty+\sup_{x\neq y}\frac{|\nabla^{\lfloor\alpha+\beta\rfloor} g(x)-\nabla^{\lfloor\alpha+\beta\rfloor} g(y)|}{|x-y|^{\alpha+\beta-\lfloor\alpha+\beta\rfloor}}.$$

Next, let $\gamma\in\scr P_k$, $\mu^i,\nu^i\in C([0,T];\scr P_k)$, $i=1,2$,
 and $\theta^{\ll,\mu^1,\nu^1}_t(x)=x+u^{\ll,\mu^1,\nu^1}_t(x)$. Take $\F_0$-measurable random variable
 $X_{0}^{\gg}$ such that $\L_{X_{0}^{\gg}}=\gg$.
 Recall that $X_{t}^{\gg,\mu^i,\nu^i}$ solves \eqref{ED00} with $s,\mu,\nu$ replaced by $0,\mu^i,\nu^i$, respectively.
 For simplicity, we denote $X_t=X_{t}^{\gg,\mu^1,\nu^1}$ and $Y_t=X_{t}^{\gg,\mu^2,\nu^2}$. Let $\tilde{N}$ be a Poisson random measure with
 compensator $\Pi(\d x)\,\d t$.
 Applying     It\^{o}'s formula (see e.g.\ \cite[Lemma 4.2]{P}) and \eqref{A10}, we have
\begin{equation*}
\begin{split}
\d \theta^{\ll,\mu^1,\nu^1}_t(X_t)&=\ll u^{\ll,\mu^1,\nu^1}_t(X_t)\,\d
t +\int_{\mathbb{R}^{m}\backslash\{0\}} \left[ \theta^{\ll,\mu^1,\nu^1}_t\left(X_{t-}+\sigma_t(\nu^1_t)x \right)
-\theta^{\ll,\mu^1,\nu^1}_t\left(X_{t-} \right) \right]\,\tilde{N}(\d x,\d t),\\
\d
\theta^{\ll,\mu^1,\nu^1}_t(Y_t)
 &=\ll
u^{\ll,\mu^1,\nu^1}_t(Y_t)\,\d t
 +\int_{\mathbb{R}^{m}\backslash\{0\}} \left[ \theta^{\ll,\mu^1,\nu^1}_t\left(Y_{t-}+\sigma_t(\nu^2_t)x \right)-\theta^{\ll,\mu^1,\nu^1}_t\left(Y_{t-}\right) \right]\,\tilde{N}(\d x,\d t)\\
&\quad+\nn\theta^{\ll,\mu^1,\nu^1}_t(Y_t)
\big[b_t(Y_t,\mu^2_t)-b_t(Y_t,\mu^1_t)\big]\,\d
t+(\A^{\nu^2}_t -\A^{\nu^1}_t) \theta^{\ll,\mu^1,\nu^1}_t(Y_t)\,\d t.
 \end{split}
\end{equation*}
By \eqref{A20} we get
$$
\frac{1}{2}|X_t-Y_t|\leq |\theta^{\lambda,\mu^1,\nu^1}_t(X_t)-\theta^{\lambda,\mu^1,\nu^1}_t(Y_t)|\leq \sum_{i=1}^5\Lambda_{i}(t),
$$
where
\beg{align*}
\Lambda_{1}(t)&:=\left|\int_{0}^{t}\int_{|x|>1}\big[\Gamma_r^{\ll,\mu^1,\nu^1}(X_{r-} ,\sigma_r(\nu^1_r)x) -\Gamma_r^{\ll,\mu^1,\nu^1}(Y_{r-} ,\sigma_r(\nu^2_r)x) \big]\,\tilde{N}(\d x,\d r)\right|,\\
\Lambda_{2}(t)&:=\left|\int_{0}^{t}\int_{0<|x|\leq 1}\big[\Gamma_r^{\ll,\mu^1,\nu^1}(X_{r-} ,\sigma_r(\nu^1_r)x) -\Gamma_r^{\ll,\mu^1,\nu^1}(Y_{r-} ,\sigma_r(\nu^2_r)x)\big]\,\tilde{N}(\d x,\d r)\right|,\\
\Lambda_{3}(t)&:=\int_{0}^{t}\lambda |u^{\ll,\mu^1,\nu^1}_r(X_r)-u^{\ll,\mu^1,\nu^1}_r(Y_r)|\,\d r,\\
\Lambda_{4}(t)&:=\int_{0}^{t}\big|\nn\theta^{\ll,\mu^1,\nu^1}_r(Y_r)
\big[b_r(Y_r,\mu^2_r)-b_r(Y_r,\mu^1_r)\big]\big|\,\d r,\\
\Lambda_{5}(t)&:=\int_{0}^{t}|(\A^{\nu^2}_r -\A^{\nu^1}_r) \theta^{\ll,\mu^1,\nu^1}_r(Y_r)|\,\d r,
\end{align*}
and
$$
    \Gamma_r^{\ll,\mu^1,\nu^1}(z,y):=\theta^{\ll,\mu^1,\nu^1}_r\left(z+ y \right)-\theta^{\ll,\mu^1,\nu^1}_r\left(z \right).
$$
Then we obtain
$$
    |X_t-Y_t|^k\leq 2^k\left(\sum_{i=1}^5\Lambda_{i}(t)\right)^k\leq 2^k\cdot5^{k-1}
    \sum_{i=1}^5\Lambda^k_{i}(t),
$$
and thus
\begin{equation}\label{upperbound}
    \mathbb{E}\left[\sup_{s\in[0,t]}|X_s-Y_s|^k\right]\leq
    C\sum_{i=1}^5\mathbb{E}\left[\sup_{s\in[0,t]}\Lambda_{i}^{k}(s)
    \right].
\end{equation}
First, by $(A2)$ and \eqref{A20}, we obtain
$$
\mathbb{E}\left[\sup_{s\in[0,t]}\Lambda_{4}^{k}(s)\right]
\leq C\left(\int_0^t[\|\mu_r^1-\mu_r^2\|_{k,var}+\W_{k}(\mu_r^1,\mu_r^2)]\,\d r\right)^k.
$$
By \eqref{A20}, we have
$$
\mathbb{E}\left[\sup_{s\in[0,t]}\Lambda_{3}^{k}(s)\right]
\leq C\int_0^t\mathbb{E}\left[\sup_{s\in[0,r]}|X_s-Y_s|^{k}\right]\d r.
$$
Moreover, \eqref{A20} and $(A3)$ imply that
\begin{align*}
    &|\Gamma_r^{\ll,\mu^1,\nu^1}(X_{r-} ,\sigma_r(\nu^1_r)x) -\Gamma_r^{\ll,\mu^1,\nu^1}(Y_{r-} ,\sigma_r(\nu^2_r)x)|\\
    &\qquad\leq|\theta_r^{\ll,\mu^1,\nu^1}(X_{r-}+\sigma_r(\nu^1_r)x)
    -\theta_r^{\ll,\mu^1,\nu^1}(Y_{r-}+\sigma_r(\nu^2_r)x)|
    +|\theta_r^{\ll,\mu^1,\nu^1}(X_{r-})-\theta_r^{\ll,\mu^1,\nu^1}(Y_{r-})|\\
    &\qquad\leq C|X_{r-}-Y_{r-}|+C|\sigma_r(\nu^1_r)x-\sigma_r(\nu^2_r)x|\\
    &\qquad\leq C|X_{r-}-Y_{r-}|+C|x|[\W_\eta(\nu^1_r,\nu^2_r)+\W_k(\nu^1_r,\nu^2_r)].
\end{align*}
This combined with the Burkholder-Davis-Gundy inequality for jump processes from \cite{Nov} (see also \cite[Theorem 3.1(i)]{KS}), implies that
\begin{align*}
\mathbb{E}\left[\sup_{s\in[0,t]}\Lambda_{1}^{k}(s)\right]&\leq  C\int_0^t\int_{|x|>1}\left\{\E[|X_{r-}-Y_{r-}|^k]+|x|^k[\W_\eta(\nu^1_r,\nu^2_r)+\W_k(\nu^1_r,\nu^2_r)]^k
\right\}\,\Pi(\d x)\d r\\
&\leq C\int_0^t\E[|X_{r-}-Y_{r-}|^k]\,\d r+C\int_0^t[\W_\eta(\nu^1_r,\nu^2_r)+\W_k(\nu^1_r,\nu^2_r)]^k\,\d r\\
&\leq C\int_0^t\E\left[\sup_{s\in[0,r]}|X_{s}-Y_{s}|^k\right]\d r+C\left(\int_0^t
[\W_\eta(\nu^1_r,\nu^2_r)+\W_k(\nu^1_r,\nu^2_r)]^2\,\d r\right)^{k/2}.
\end{align*}
Next, by \cite[Lemma 4.1]{P}, \eqref{A20}, \eqref{uuu} and $(A3)$, we get
\begin{align*}&|\Gamma_r^{\ll,\mu^1,\nu^1}(X_{r-} ,\sigma_r(\nu^1_r)x) -\Gamma_r^{\ll,\mu^1,\nu^1}(Y_{r-} ,\sigma_r(\nu^2_r)x)|\\
&\qquad\qquad\qquad\quad\leq |\Gamma_r^{\ll,\mu^1,\nu^1}(X_{r-} ,\sigma_r(\nu^1_r)x) -\Gamma_r^{\ll,\mu^1,\nu^1}(Y_{r-} ,\sigma_r(\nu^1_r)x)|\\
&\qquad\qquad\qquad\qquad+|\Gamma_r^{\ll,\mu^1,\nu^1}(Y_{r-} ,\sigma_r(\nu^1_r)x) -\Gamma_r^{\ll,\mu^1,\nu^1}(Y_{r-} ,\sigma_r(\nu^2_r)x)|\\
&\qquad\qquad\qquad\quad\leq C|x|^{\alpha+\beta-1}|X_{r-}-Y_{r-}|+C|x|[\W_\eta(\nu^1_r,\nu^2_r)+\W_k(\nu^1_r,\nu^2_r)].
\end{align*}
Since $2\alpha+2\beta-2=\alpha+(\alpha+2\beta-2)>\alpha$ and
\begin{align}\label{sma}\int_{0<|x|\leq 1}|x|^\epsilon\,\Pi(\d x)<\infty\quad \text{for all $\epsilon>\alpha$},
\end{align}
it holds that
\begin{align*}
&C\mathbb{E}\left[\sup_{s\in[0,t]}\Lambda_{2}^{k}(s)\right]\\
&\leq C\E \left(\int_{0}^{t}\int_{0<|x|\leq 1}|\Gamma_r^{\ll,\mu^1,\nu^1}(X_{r-} ,\sigma_r(\nu^1_r)x) -\Gamma_r^{\ll,\mu^1,\nu^1}(Y_{r-} ,\sigma_r(\nu^1_r)x)|^2\,\Pi(\d x)\d r\right)^{k/2}\\
&\leq C\E \left(\int_{0}^{t}\int_{0<|x|\leq 1}\left[|x|^{2\alpha+2\beta-2}|X_{r-}-Y_{r-}|^2+|x|^2[\W_\eta(\nu^1_r,\nu^2_r)+\W_k(\nu^1_r,\nu^2_r)]^2\right]\,\Pi(\d x)\d r\right)^{k/2}\\
&\leq C\E
\left(
\int_0^t|X_{r}-Y_{r}|^2\,\d r
\right)^{k/2}+C\left(\int_0^t[\W_\eta(\nu^1_r,\nu^2_r)+\W_k(\nu^1_r,\nu^2_r)]^2\,\d r\right)^{k/2}\\
&\leq \frac{1}{2}\,\E\left[\sup_{s\in[0,t]}|X_{s}-Y_{s}|^k\right]+C\int_0^t\E\left[\sup_{s\in[0,r]}|X_{s}-Y_{s}|^k\right]\d r\\
&\quad
+C\left(\int_0^t[\W_\eta(\nu^1_r,\nu^2_r)+\W_k(\nu^1_r,\nu^2_r)]^2\,\d r\right)^{k/2}.
\end{align*}
Here in the last inequality we have used the fact that if $f:[0,T]\rightarrow[0,\infty)$ is a right continuous
function having left limits, then
\begin{align*}
    C\left(\int_0^tf(r)^2\,\d r\right)^{k/2}
    &\leq \sup_{s\in[0,t]}f(s)^{k/2}\times C\left(\int_0^tf(r)\,\d r\right)^{k/2}\\
    &\leq\frac12\,\sup_{s\in[0,t]}f(s)^{k}+\frac12\,C^2\left(\int_0^tf(r)\,\d r\right)^{k}\\
    &\leq\frac12\,\sup_{s\in[0,t]}f(s)^{k}+C\int_0^tf(r)^k\,\d r\\
    &\leq\frac12\,\sup_{s\in[0,t]}f(s)^{k}+C\int_0^t\sup_{s\in[0,r]}f(s)^k\,\d r.
\end{align*}
Finally, it follows from $(A3)$, \eqref{A20}, \eqref{uuu} and \eqref{sma} that
\begin{align*}
&\mathbb{E}\left[\sup_{s\in[0,t]}\Lambda_{5}^{k}(s)\right]\\
&\leq\E\bigg(\int_0^t\int_{\mathbb{R}^{m}\setminus\{0\}}\Big|\theta^{\ll,\mu^1,\nu^1}_r(Y_r+\sigma_r(\nu^1_r)y) -\theta^{\ll,\mu^1,\nu^1}_r(Y_r+\sigma_r(\nu^2_r)y)\\
&\qquad\qquad\qquad\quad-\langle \sigma_r(\nu^1_r)y-\sigma_r(\nu^2_r)y,\nabla \theta^{\ll,\mu^1,\nu^1}_r(Y_r)
\rangle\mathds{1}_{\{ |y|\leq 1\}} \Big|\,\Pi(\d y)\d r\bigg)^k\\
&\leq C\left(\int_{|y|>1}|y|\,\Pi(\d y)+\int_{0<|y|\leq 1}|y|^{\alpha+\beta}\,\Pi(\d y)\right)^k\left(\int_{0}^t[\W_\eta(\nu^1_r,\nu^2_r)+\W_k(\nu^1_r,\nu^2_r)]\,\d r\right)^k\\
&\leq C\left(\int_{0}^t[\W_\eta(\nu^1_r,\nu^2_r)+\W_k(\nu^1_r,\nu^2_r)]\,\d r\right)^k\\
&\leq C\left(\int_{0}^t[\W_\eta(\nu^1_r,\nu^2_r)+\W_k(\nu^1_r,\nu^2_r)]^2\,\d r\right)^{k/2}.
\end{align*}
Combining \eqref{upperbound} and the above estimates, we get
\begin{align*}
    \mathbb{E}\left[\sup_{s\in[0,t]}|X_s-Y_s|^k\right]
    &\leq C\int_0^t\mathbb{E}\left[\sup_{s\in[0,r]}|X_s-Y_s|^{k}\right]\d r\\
    &\quad+C\left(\int_0^t[\|\mu_r^1-\mu_r^2\|_{k,var}+\W_{k}(\mu_r^1,\mu_r^2)]\,\d r\right)^k\\
    &\quad+C\left(\int_{0}^t[\W_\eta(\nu^1_r,\nu^2_r)+\W_k(\nu^1_r,\nu^2_r)]^2\,\d r\right)^{k/2}.
\end{align*}
This, together with Gronwall's inequality, implis
\begin{align*}
    \W_k(\L_{X_t},\L_{Y_t})^k&\leq\mathbb{E}\left[\sup_{s\in[0,t]}|X_s-Y_s|^k\right]\\
    &\leq
    C\left(\int_0^t[\|\mu_r^1-\mu_r^2\|_{k,var}+\W_{k}(\mu_r^1,\mu_r^2)]\,\d r\right)^k\\
    &\quad +C\left(\int_{0}^t[\W_\eta(\nu^1_r,\nu^2_r)+\W_k(\nu^1_r,\nu^2_r)]^2\,\d r\right)^{k/2}.
\end{align*}
Then for any $\delta>0$ and $t\in[0,T]$,
\begin{align*}
    \e^{-\delta t}\W_k(\L_{X_t},\L_{Y_t})
    &\leq C\e^{-\delta t}\int_0^t\e^{-\delta r}[\|\mu_r^1-\mu_r^2\|_{k,var}+\W_{k}(\mu_r^1,\mu_r^2)]\cdot\e^{\delta r}\,\d r\\
    &\quad+C\e^{-\delta t}\left(\int_{0}^t\e^{-2\delta r}[\W_\eta(\nu^1_r,\nu^2_r)+\W_k(\nu^1_r,\nu^2_r)]^2\cdot\e^{2\delta r}\,\d r\right)^{1/2}\\
    &\leq C\sup_{r\in[0,T]}\e^{-\delta r}[\|\mu_r^1-\mu_r^2\|_{k,var}+\W_{k}(\mu_r^1,\mu_r^2)]
    \times\int_0^t\e^{-\delta (t-s)}\,\d s\\
    &\quad+C\sup_{r\in[0,T]}\e^{-\delta r}[\W_\eta(\nu^1_r,\nu^2_r)+\W_k(\nu^1_r,\nu^2_r)]
    \times\left(\int_{0}^t\e^{-2\delta (t-s)}\,\d s\right)^{1/2}\\
    &\leq \frac{C}{\delta}\sup_{r\in[0,T]}\e^{-\delta r}[\|\mu_r^1-\mu_r^2\|_{k,var}+\W_{k}(\mu_r^1,\mu_r^2)]\\
    &\quad+\frac{C}{\sqrt{2\delta}}\sup_{r\in[0,T]}\e^{-\delta r}[\W_\eta(\nu^1_r,\nu^2_r)+\W_k(\nu^1_r,\nu^2_r)],
\end{align*}
 which completes the proof.
 \end{proof}

\subsection{The method of time-change}

It is well-known that an $m$-dimensional rotationally symmetric $\alpha$-stable L\'evy process $Z_t$ can be represented
as subordinated Brownian motion, see for instance \cite{SSZ12}.  More precisely, let $S_t$ be an $\frac{\alpha}{2}$-stable
subordinator, i.e.\ $S_t$ is a $[0,\infty)$-valued L\'evy process with the following Laplace transform:
$$
    \E\left[\e^{-rS_t}\right] = \e^{-2^{-1}t (2r)^{\alpha/2}},\quad r>0,\,t\geq 0,
$$
and let $W_t$ be an $m$-dimensional standard Brownian motion, which is independent of $S_t$.
The time-changed process $Z_{t}:=W_{S_{t}}$ is an $m$-dimensional rotationally symmetric $\alpha$-stable L\'evy process
such that $\E\,\e^{\iup \<\xi, Z_t\>}=\e^{-t|\xi|^\alpha/2}$ for $\xi\in\R^m$.
Using the subordination representation, \eqref{E1} can be written in the following form
$$
\d X_t=b_t(X_t,\L_{X_t})\,\d t+\sigma_t(\L_{X_t})\,\d W_{S_t},\quad t\in[0,T].
$$

For $ 0\leq s< t\leq T$, $x\in\R^d$, and $\nu\in C([0,T];\scr P_k)$, let $q_{s,t}^{\nu}(x,\cdot)$ be the density
function of the random variable
$$
    Y_{s,t}^{x,\nu}:= x + \int_s^t \si_r(\nu_r)\,\d W_{S_r}.
$$
As before, we write $q_{t}^{\nu}(x,\cdot)=q_{0,t}^{\nu}(x,\cdot)$ for $t>0$.
Since $W_t$ is independent of $S_t$, one has
\beq\label{ES6'}
    q_{s,t}^{\nu}(x,y)=\E\,q_{s,t}^{\nu,S}(x,y),
\end{equation}
where
$$
    q_{s,t}^{\nu,S}(x,y):=\ff{\exp\left[-\ff 1 {2} \left\<(a_{s,t}^{\nu,S})^{-1}(y-x),  y-x\right\>\right]}{(2\pi )^{d/2}
    ({\rm det} \{a_{s,t}^{\nu,S}\})^{1/2}}\quad \text{and}\quad
    a_{s,t}^{\nu,S} :=  \int_s^t(\si_r\si_r^*)( \nu_r)\,\d S_r.
$$
Obviously, $(A3)$   implies
\begin{equation}\label{mv}
    \|a_{s,t}^{\nu^1,S}-a_{s,t}^{\nu^2,S}\|\le 2K_2^{3/2}\int_s^t[\W_k(\nu^1_r,\nu^2_r)+\W_\eta(\nu^1_r,\nu^2_r)]\,\d S_r,
    \quad \nu^1,\nu^2\in C([0,T];\scr P_k),
\end{equation}
\begin{equation}\label{mv2}
    \frac{1}{K_2(S_t-S_s)}\leq \|[a_{s,t}^{\nu,S}]^{-1}\|\leq \frac{K_2}{S_t-S_s},
    \quad \nu\in C([0,T];\scr P_k).
\end{equation}

For $0\leq s< t\leq T$ and $x,y\in \R^d$, let
\begin{align*}
q_{s,t}^{\nu,S}(x,y):=\ff{\exp\left[-\ff 1 {2} \left\<(a_{s,t}^{\nu,S})^{-1}(y-x),  y-x\right\>\right]}{(2\pi )^{d/2}
({\rm det} \{a_{s,t}^{\nu,S}\})^{1/2}}\quad \text{and}\quad
\tilde{q}_{s,t}^{S}(x,y):=\ff{\exp\left[-\frac{|y-x|^2}{4K_2(S_t-S_s)}\right]}{(4K_2\pi(S_t-S_s))^{d/2}}.
\end{align*}
By \eqref{mv} and  \eqref{mv2}, one can apply the argument used in the
proof of \cite[Lemma 3.1]{HWJMAA} to
get the following lemma. To save space, we omit the proof.

\beg{lem}\label{L20}  Assume $(A3)$. For any $0\leq s<t\le T$, $x,y\in \R^d$ and $
\nu^1,\nu^2,\nu\in C([0,T];\scr P_k)$,
\begin{align*}
    &|q_{s,t}^{\nu^1,S}(x,y)-q_{s,t}^{\nu^2,S}(x,y)|\leq C  \tilde{q}^{S}_{s,t}(x,y)
    (S_t-S_s)^{-1}\int_s^t[\W_\eta(\nu^1_r,\nu^2_r)+\W_k(\nu^1_r,\nu^2_r)]\,\d S_r,\\
    &|\nabla q_{s,t}^{\nu,S}(\cdot,y)(x)|\leq C \tilde{q}^{S}_{s,t}(x,y)(S_t-S_s)^{-1/2},\\
    &|\nabla q_{s,t}^{\nu^1,S}(\cdot,y)(x)-\nabla q_{s,t}^{\nu^2,S}(\cdot,y)(x)|\leq
      C\tilde{q}^{S}_{s,t}(x,y)(S_t-S_s)^{-3/2}
    \int_s^t[\W_\eta(\nu^1_r,\nu^2_r)+\W_k(\nu^1_r,\nu^2_r)]\,\d S_r.
\end{align*}
\end{lem}

\beg{lem}\label{L2}  Assume $(A3)$.  For any $0\leq s<t\le T$, $x\in \R^d$, $
\nu^1,\nu^2,\nu\in C([0,T];\scr P_k)$ and $\epsilon\in[0,\alpha)$,
\begin{equation}\label{g0'}
\begin{aligned}
&\int_{\R^d}|q_{s,t}^{\nu^1}(x,y)-q_{s,t}^{\nu^2}(x,y)||y-x|^\epsilon\,\d y\\
&\qquad\qquad\;\leq C \E\left[(S_t-S_s)^{-1+\epsilon/2}\int_s^t[\W_\eta(\nu^1_r,\nu^2_r)+\W_k(\nu^1_r,\nu^2_r)]\,\d S_r
\right],
\end{aligned}
\end{equation}
\begin{equation}\label{g1'}
\begin{aligned}
\int_{\R^d}|\nabla q_{s,t}^{\nu}(\cdot,y)(x)||y-x|^\epsilon\,\d y&\leq C (t-s)^{(-1+\epsilon)/\alpha},
\end{aligned}
\end{equation}
\begin{equation}\label{g2'}
\begin{aligned}
&\int_{\R^d}|\nabla q_{s,t}^{\nu^1}(\cdot,y)(x)-\nabla q_{s,t}^{\nu^2}(\cdot,y)(x)||y-x|^\epsilon\,\d y\\
&\qquad\qquad\leq
 C \E\left[(S_t-S_s)^{(-3+\epsilon)/2}
 \int_s^t[\W_\eta(\nu^1_r,\nu^2_r)+\W_k(\nu^1_r,\nu^2_r)]\,\d S_r
 \right].
\end{aligned}
\end{equation}
\end{lem}
\begin{proof}
It follows from \eqref{ES6'} and the first inequality in Lemma \ref{L20} that for all $\epsilon\in[0,\alpha)$,
\begin{align*}
&\int_{\R^d}| q_{s,t}^{\nu^1}(x,y)- q_{s,t}^{\nu^2}(x,y)||y-x|^\epsilon\,\d y\\
&\leq C\E\left[(S_t-S_s)^{-1}\int_s^t[\W_\eta(\nu^1_r,\nu^2_r)+\W_k(\nu^1_r,\nu^2_r)]\,\d S_r
\times\int_{\R^d} \tilde{q}^{S}_{s,t} (x,y)  |y-x|^\epsilon\,\d y\right]\\
&=C\E\left[
(S_t-S_s)^{-1+\epsilon/2}\int_s^t[\W_\eta(\nu^1_r,\nu^2_r)+\W_k(\nu^1_r,\nu^2_r)]\,\d S_r
\right]\times\int_{\R^d} |y|^\epsilon\ff{\exp\left[-\frac{|y|^2}{4K_2}\right]}{(4K_2\pi)^{d/2}}\,\d y\\
&\leq C\E\left[
(S_t-S_s)^{-1+\epsilon/2}\int_s^t[\W_\eta(\nu^1_r,\nu^2_r)+\W_k(\nu^1_r,\nu^2_r)]\,\d S_r
\right].
\end{align*}
This is what we have claimed in \eqref{g0'}. We can prove \eqref{g2'} in a similar way.
To prove \eqref{g1'}, we use \eqref{ES6'} and the second inequality in Lemma \ref{L20} to get
\begin{align*}
\int_{\R^d}|\nabla q_{s,t}^{\nu}(\cdot,y)(x)||y-x|^\epsilon\,\d y
&\leq C\E\left[(S_t-S_s)^{-1/2}\int_{\R^d} \tilde{q}^{S}_{s,t} (x,y)  |y-x|^\epsilon\,\d y\right]\\
&=C\E\left[(S_t-S_s)^{(-1+\epsilon)/2}\right]\times\int_{\R^d} |y|^\epsilon\ff{\exp\left[-\frac{|y|^2}{4K_2}\right]}
{(4K_2\pi)^{d/2}}\,\d y\\
&\leq C\E\left[S_{t-s}^{(-1+\epsilon)/2}\right].
\end{align*}
Since by \cite[Lemma 4.1]{DS19},
$$
    \E\left[S_{t-s}^{(-1+\epsilon)/2}\right]
    =\frac{\Gamma\left(1-\frac{-1+\epsilon}{\alpha}\right)}{\Gamma\left(\frac{3-\epsilon}{2}\right)}\,(t-s)^{(-1+\epsilon)/\alpha}
    \leq C(t-s)^{(-1+\epsilon)/\alpha}
    \quad \text{for all $\epsilon\in[0,\alpha)$},
$$
this completes the proof.
\end{proof}

Recall that
$$
    X_{s,t}^{x,\mu,\nu}=x+\int_s^tb_r(X_{s,r}^{x,\mu,\nu},\mu_r)\,\d r
    +\int_s^t\sigma_r(\nu_r)\,\d W_{S_r},\quad x\in\R^d,\,0\leq s< t\leq T.
$$
Since $b$ and $\sigma$ are bounded due to $(A2)$ and $(A3)$, it follows from
$\E [S_T^{k/2}]=T^{k/\alpha}\E [S_1^{k/2}]<\infty$ that
\beg{equation}\begin{split}\label{gr4e23}
   \E\left[\sup_{t\in[s,T]}|X_{s,t}^{x,\mu,\nu}|^k\right]
   &\leq C|x|^k+C+C\E\left[\sup_{t\in[s,T]}
   \left|
   \int_s^t\sigma_r(\nu_r)\,\d W_{S_r}
   \right|^k
   \right]\\
   &\leq C|x|^k+C+C\E [S_T^{k/2}]\\
   &\leq C(1+|x|^k).
\end{split}
\end{equation}
This implies
\begin{equation}\label{NNT}
    \E\left[1+|X_{s,t}^{x,\mu,\nu}|^k\right]
    \leq C (1+|x|^k),\quad \mu,\nu\in C([0,T];\scr P_k).
\end{equation}

We denote by $p_{s,t}^{\mu,\nu}(x,\cdot)$ the
density function of $X_{s,t}^{x,\mu,\nu}$. Denote by $P_{s,t}^{\mu,\nu}$ and $Q_{s,t}^\nu$ the inhomogeneous Markov semigroups associated
with $X_{s,t}^{x,\mu,\nu}$ and $Y_{s,t}^{x,\nu}$, respectively, i.e.\ for $f\in \scr B_b(\R^d)$,
\begin{align*}
    P_{s,t}^{\mu,\nu}f(x)&=\E f(X_{s,t}^{x,\mu,\nu})=\int_{\R^d}p_{s,t}^{\mu,\nu}(x,y)f(y)\,\d y,\\
Q_{s,t}^\nu f(x)&=\E f(Y_{s,t}^{x,\nu})=\int_{\R^d}q_{s,t}^{\nu}(x,y)f(y)\,\d y.
 \end{align*}
As before, write $p_{t}^{\mu,\nu}(x,\cdot)=p_{0,t}^{\mu,\nu}(x,\cdot)$,
$P_{t}^{\mu,\nu}=P_{0,t}^{\mu,\nu}$ and $Q_{t}^\nu=Q_{0,t}^\nu$ for $t>0$.

\begin{lem}\label{dunh4}
Assume $(A2)$ and $(A3)$. Then for any $0\leq s<t\leq T$, $\mu,\nu\in C([0,T];\scr P_k)$,
and $f\in\scr B_b(\R^d)$,
$$
P_{s,t}^{\mu,\nu}f = Q_{s,t}^{\nu}f + \int_s^t  P_{s,r}^{\mu,\nu}
\left\<b_r(\cdot, \mu_r),\nabla Q_{r,t}^{\nu}f\right\>\d r.
$$
\end{lem}
\begin{proof}
By a standard approximation argument, it suffices to prove the desired assertion
for $f\in C_b^2(\R^d)$. By the backward Kolmogorov equation, it holds that
$$
\frac{\partial Q_{r,t}^{\nu}f}{\partial r}
=-\A_r^{\nu}(Q_{r,t}^{\nu}f),\quad 0\leq r< t\leq T,
$$
where $\A^{\nu}_r f$ is given by \eqref{L34}.
Similarly, we have the forward Kolmogorov equation
$$\frac{\partial P_{s,r}^{\mu,\nu}f}{\partial r}
=P_{s,r}^{\mu,\nu} [\A_{r}^{\mu,\nu}f],\quad  0\leq s< r\leq T,$$
where
$$
\A^{\mu,\nu}_r f(x):=\<b_t(x,\mu_r),\nabla f(x)\>+\A^{\nu}_r f(x).
$$
Hence, we have
\begin{align*}
    P_{s,t}^{\mu,\nu}f -Q_{s,t}^{\nu}f
    &=\int_s^t
    \frac{\partial}{\partial r}[P_{s,r}^{\mu,\nu}Q^{\nu}_{r,t}f]\,\d r\\
    &=\int_s^t
   P_{s,r}^{\mu,\nu}\{[\A_{r}^{\mu,\nu}-\A_{r}^{\nu}]Q^{\nu}_{r,t}f\}
    \,\d r\\
    &=\int_s^t  P_{s,r}^{\mu,\nu} \left\<b_r(\cdot, \mu_r),\nabla Q_{r,t}^{\nu}f\right\>\d r,
\end{align*}
and this completes the proof.
\end{proof}

\subsection{An auxiliary lemma}

\begin{lem}\label{vdh22s}
Assume $(A1)$-$(A3)$. Let $\gg\in \scr P_k$ and $\mu^i,\nu^i\in C([0,T];\scr P_k)$, $i=1,2$.
Then for large enough $\delta>0$,
\begin{align*}
    &\sup_{t\in[0,T]}\e^{-\delta t}\big\|\L_{{X}_{t}^{\gg,\mu^1,\nu^1}}-\L_{{X}_{t}^{\gg,\mu^2,\nu^2}}\big\|_{k,var}
    \leq C\gamma(1+|\cdot|^k)\sup_{t\in[0,T]}\e^{-\delta t}[\W_\eta(\nu^1_t,\nu^2_t)+\W_k(\nu^1_t,\nu^2_t)]\\
    &\quad\qquad\qquad\qquad\qquad\qquad+C\gamma(1+|\cdot|^k)\delta^{1/\alpha-1}
    \sup_{t\in[0,T]}\e^{-\delta t}\left[
    \|\mu^1_t-\mu^2_t\|_{k,var}+\W_{k}(\mu^1_t,\mu^2_t)
    \right],\\
    &\sup_{t\in[0,T]}\e^{-\delta t}\W_\eta\big(\L_{{X}_{t}^{\gg,\mu^1,\nu^1}},\L_{{X}_{t}^{\gg,\mu^2,\nu^2}}\big)
    \leq \frac14\sup_{t\in[0,T]}\e^{-\delta t}[\W_\eta(\nu^1_t,\nu^2_t)+\W_k(\nu^1_t,\nu^2_t)]\\
    &\quad\qquad\qquad\qquad\qquad\qquad+C\gamma(1+|\cdot|^k)\delta^{1/\alpha-1}
    \sup_{t\in[0,T]}\e^{-\delta t}\left[
    \|\mu^1_t-\mu^2_t\|_{k,var}+\W_{k}(\mu^1_t,\mu^2_t)
    \right].
\end{align*}
\end{lem}

\begin{proof}
By Lemma \ref{dunh4} with $s=0$, we obtain that for $t>0$ and $f\in\scr B_b(\R^d)$,
\begin{align*}
&\int_{\R^d}\big[
    P_{t}^{\mu^1,\nu^1}f(x)-P_{t}^{\mu^2,\nu^2}f(x)\big]\,\gg(\d x)\\
&=\int_{\R^d} \big[Q_{t}^{\nu^1}f(x)-Q_{t}^{\nu^2}f(x)\big] \,\gg(\d x)\\
&\quad+\int_{\R^d}\gg(\d x)\int_0^t \big[ P_{r}^{\mu^1,\nu^1} \left\<b_r(\cdot, \mu^1_r),\nabla Q_{r,t}^{\nu^1}f\right\>(x)-P_{r}^{\mu^2,\nu^2} \left\<b_r(\cdot, \mu^2_r),\nabla Q_{r,t}^{\nu^2}f\right\>(x)\big]\d r\\
&=\int_{\R^d} \big[Q_{t}^{\nu^1}f(x)-Q_{t}^{\nu^2}f(x)\big] \,\gg(\d x) \\
&\quad+\int_{\R^d}\gg(\d x)\int_0^t \big[ P_{r}^{\mu^1,\nu^1} \left\<b_r(\cdot, \mu^1_r),\nabla Q_{r,t}^{\nu^1}f\right\>(x)-P_{r}^{\mu^2,\nu^2} \left\<b_r(\cdot, \mu^1_r),\nabla Q_{r,t}^{\nu^1}f\right\>(x)\big]\d r\\
&\quad+\int_{\R^d}\gg(\d x)\int_0^t  P_{r}^{\mu^2,\nu^2} \left\<b_r(\cdot, \mu^1_r)-b_r(\cdot, \mu^2_r),
\nabla Q_{r,t}^{\nu^1}f\right\>(x)\,\d r\\
&\quad+\int_{\R^d}\gg(\d x)\int_0^t  P_{r}^{\mu^2,\nu^2} \left\<b_r(\cdot, \mu^2_r),\nabla Q_{r,t}^{\nu^1}f-\nabla Q_{r,t}^{\nu^2}f
\right\>(x)\,\d r\\
&=:\mathsf{J}_1+\mathsf{J}_2+\mathsf{J}_3+\mathsf{J}_4.
\end{align*}
Note that
\begin{align*}
    \big\|\L_{{X}_{t}^{\gg,\mu^1,\nu^1}}-\L_{{X}_{t}^{\gg,\mu^2,\nu^2}}\big\|_{k,var}
    &=\sup_{f\in\scr B_b(\R^d),|f|\leq 1+|\cdot|^k}\left|
    \int_{\R^d}\big[
    P_{t}^{\mu^1,\nu^1}f(x)-P_{t}^{\mu^2,\nu^2}f(x)\big]\,\gg(\d x)\right|\\
    &\leq\sum_{i=1}^4\sup_{f\in\scr B_b(\R^d),|f|\leq 1+|\cdot|^k}|\mathsf{J}_i|.
\end{align*}
We will estimate the terms $|\mathsf{J}_i|$, $i=1,\dots,4$ separately. First, it follows from \eqref{g0'}
with $\epsilon=k$ and $\epsilon=0$ that for all $t\in(0,T]$, $\delta>0$ and $f\in\scr B_b(\R^d)$ satisfying $|f|\leq 1+|\cdot|^k$,
\begin{align*}
    |\mathsf{J}_1|&=\left|\int_{\R^d}\gg(\d x)
    \int_{\R^d}\big[q_{t}^{\nu^1}(x,y)-q_{t}^{\nu^2}(x,y)\big]f(y)\,\d y
    \right|\\
    &\leq \int_{\R^d}\gg(\d x)\int_{\R^d}\big|q_{t}^{\nu^1}(x,y)-q_{t}^{\nu^2}(x,y)\big|(1+|y|^k)\,\d y\\
    &\leq C\int_{\R^d}\gg(\d x)\int_{\R^d}\big|q_{t}^{\nu^1}(x,y)-q_{t}^{\nu^2}(x,y)\big|(1+|x|^k+|y-x|^k)\,\d y\\
    &\leq C\gamma(1+|\cdot|^k)\E\left[
    \big(S_t^{-1}+S_t^{-1+k/2}\big)\int_0^t[\W_\eta(\nu^1_r,\nu^2_r)+\W_k(\nu^1_r,\nu^2_r)]\,\d S_r
    \right]\\
    &\leq C\gamma(1+|\cdot|^k)
    \sup_{s\in[0,t]}[\W_\eta(\nu^1_s,\nu^2_s)+\W_k(\nu^1_s,\nu^2_s)]
    \times\E\left[
    \big(S_t^{-1}+S_t^{-1+k/2}\big)S_t
    \right]\\
    &\leq C\gamma(1+|\cdot|^k)\e^{\delta t}
    \sup_{s\in[0,t]}\e^{-\delta s}[\W_\eta(\nu^1_s,\nu^2_s)+\W_k(\nu^1_s,\nu^2_s)]
    \times\left(
    1+\E\big[S_T^{k/2}\big]
    \right)\\
    &\leq C\gamma(1+|\cdot|^k)\e^{\delta t}
    \sup_{s\in[0,T]}\e^{-\delta s}[\W_\eta(\nu^1_s,\nu^2_s)+\W_k(\nu^1_s,\nu^2_s)].
\end{align*}
Now we turn to $|\mathsf{J}_2|$. Note that
\begin{align*}
    |\mathsf{J}_2|
    &=\left|\int_{\R^d}\gg(\d x)\int_0^t\d r
    \int_{\R^d} \big[p_{r}^{\mu^1,\nu^1}(x,y) -p_{r}^{\mu^2,\nu^2}(x,y) \big]
    \left\<b_r(y, \mu^1_r),
    \int_{\R^d}\nabla q_{r,t}^{\nu^1}(\cdot,z)(y)f(z)\,\d z\right\>\d y\right|\\
    &\leq\int_0^t\left|
    \int_{\R^d}\gg(\d x)\int_{\R^d} \big[p_{r}^{\mu^1,\nu^1}(x,y) -p_{r}^{\mu^2,\nu^2}(x,y) \big]
    \left\<b_r(y, \mu^1_r),
    \int_{\R^d}\nabla q_{r,t}^{\nu^1}(\cdot,z)(y)f(z)\,\d z\right\>\d y
    \right|\d r.
\end{align*}
Since it follows from $(A2)$ and \eqref{g1'} with $\epsilon=k$ and $\epsilon=0$ that
\begin{align*}
    \left|\left\<b_r(y, \mu^1_r),
    \int_{\R^d}\nabla q_{r,t}^{\nu^1}(\cdot,z)(y)f(z)\,\d z\right\>\right|
    &\leq \|b\|_\infty\int_{\R^d}\big|\nabla q_{r,t}^{\nu^1}(\cdot,z)(y)\big|(1+|z|^k)\,\d z\\
    &\leq C(t-r)^{-1/\alpha}(1+|y|^k),
\end{align*}
we have for all $t\in(0,T]$, $\delta>0$ and $f\in\scr B_b(\R^d)$ satisfying $|f|\leq 1+|\cdot|^k$,
\begin{align*}
    |\mathsf{J}_2|
    &\leq C\int_0^t(t-r)^{-1/\alpha}
    \big\|\L_{{X}_{r}^{\gg,\mu^1,\nu^1}}-\L_{{X}_{r}^{\gg,\mu^2,\nu^2}}\big\|_{k,var}
    \,\d r\\
    &=C\e^{\delta t}\int_0^t\e^{-\delta r}
    \big\|\L_{{X}_{r}^{\gg,\mu^1,\nu^1}}-\L_{{X}_{r}^{\gg,\mu^2,\nu^2}}\big\|_{k,var}\cdot(t-r)^{-1/\alpha}
    \e^{-\delta (t-r)}
    \,\d r\\
    &\leq C\e^{\delta t}\sup_{s\in[0,T]}\e^{-\delta s}\big\|\L_{{X}_{s}^{\gg,\mu^1,\nu^1}}-\L_{{X}_{s}^{\gg,\mu^2,\nu^2}}\big\|_{k,var}
    \times\int_0^t
    (t-r)^{-1/\alpha}
    \e^{-\delta (t-r)}
    \,\d r\\
    &\leq C\delta^{1/\alpha-1}\e^{\delta t}\sup_{s\in[0,T]}\e^{-\delta s}\big\|\L_{{X}_{s}^{\gg,\mu^1,\nu^1}}-\L_{{X}_{s}^{\gg,\mu^2,\nu^2}}\big\|_{k,var},
\end{align*}
where in the last equality we have used the following fact
\begin{equation}\label{expint}
    \sup_{t\in[0,T]}
    \int_0^t(t-r)^{-1/\alpha}\e^{-\delta(t-r)}\,\d r
    \leq\int_0^\infty r^{-1/\alpha}\e^{-\delta r}\,\d r
    =\Gamma\left(1-\frac{1}{\alpha}\right)\delta^{1/\alpha-1}
\end{equation}
By $(A2)$, \eqref{g1'} with $\epsilon=k$ and $\epsilon=0$, \eqref{NNT} and \eqref{expint}, we derive
for all $t\in(0,T]$, $\delta>0$ and $f\in\scr B_b(\R^d)$ satisfying $|f|\leq 1+|\cdot|^k$,
\begin{align*}
    |\mathsf{J}_3|&=\left|\int_{\R^d}\gg(\d x)\int_0^t\d r
    \int_{\R^d}
    p_{r}^{\mu^2,\nu^2} (x,y)\left\<b_r(y, \mu^1_r)-b_r(y, \mu^2_r),
    \int_{\R^d}\nabla q_{r,t}^{\nu^1}(\cdot,z)(y)f(z)\,\d z\right\>\d y
    \right|\\
    &\leq\int_{\R^d}\gg(\d x)\int_0^t\d r\int_{\R^d}p_{r}^{\mu^2,\nu^2} (x,y)
    |b_r(y, \mu^1_r)-b_r(y, \mu^2_r)|\,\d y
    \int_{\R^d}\big|
    \nabla q_{r,t}^{\nu^1}(\cdot,z)(y)\big|(1+|z|^k)\,\d z\\
    &\leq C\int_{\R^d}\gg(\d x)\int_0^t(t-r)^{-1/\alpha}\left[
    \|\mu^1_r-\mu^2_r\|_{k,var}+\W_{k}(\mu^1_r,\mu^2_r)
    \right]\d r
    \int_{\R^d}p_{r}^{\mu^2,\nu^2} (x,y)(1+|y|^k)\,\d y\\
    &\leq C\int_{\R^d}(1+|x|^k)\,\gg(\d x)\int_0^t(t-r)^{-1/\alpha}\left[
    \|\mu^1_r-\mu^2_r\|_{k,var}+\W_{k}(\mu^1_r,\mu^2_r)
    \right]\d r\\
    &\leq C\gamma(1+|\cdot|^k)\e^{\delta t}\sup_{s\in[0,T]}\e^{-\delta s}\left[
    \|\mu^1_s-\mu^2_s\|_{k,var}+\W_{k}(\mu^1_s,\mu^2_s)
    \right]
    \times\int_0^t(t-r)^{-1/\alpha}\e^{-\delta(t-r)}\,\d r\\
    &\leq C\gamma(1+|\cdot|^k)\delta^{1/\alpha-1}\e^{\delta t}\sup_{s\in[0,T]}\e^{-\delta s}\left[
    \|\mu^1_s-\mu^2_s\|_{k,var}+\W_{k}(\mu^1_s,\mu^2_s)
    \right].
\end{align*}
It holds from $(A2)$, \eqref{g2'} with $\epsilon=k$ and $\epsilon=0$, and \eqref{NNT} that
for all $t\in(0,T]$, $\delta>0$ and $f\in\scr B_b(\R^d)$ satisfying $|f|\leq 1+|\cdot|^k$,
\begin{align*}
    |\mathsf{J}_4|&=\left|\int_{\R^d}\gg(\d x)\int_0^t\d r
    \int_{\R^d}p_{r}^{\mu^2,\nu^2} (x,y)
    \left\<b_r(y, \mu^2_r),
    \int_{\R^d}\big[\nabla q_{r,t}^{\nu^1}(\cdot,z)(y)-\nabla q_{r,t}^{\nu^2}(\cdot,z)(y)\big]
    f(z)\,\d z\right\>\d y
    \right|\\
    &\leq\|b\|_\infty\int_{\R^d}\gg(\d x)\int_0^t\d r\int_{\R^d}p_{r}^{\mu^2,\nu^2} (x,y)\,\d y
    \int_{\R^d}\big|\nabla q_{r,t}^{\nu^1}(\cdot,z)(y)-\nabla q_{r,t}^{\nu^2}(\cdot,z)(y)\big|
    (1+|z|^k)\,\d z\\
    &\leq C\int_{\R^d}\gg(\d x)\int_0^t\E\left[
    \left(
    (S_t-S_r)^{(k-3)/2}+(S_t-S_r)^{-3/2}
    \right)
    \int_r^t[\W_\eta(\nu^1_s,\nu^2_s)+\W_k(\nu^1_s,\nu^2_s)]\,\d S_s
    \right]\d r\\
    &\quad\times
    \int_{\R^d}p_{r}^{\mu^2,\nu^2} (x,y)(1+|y|^k)\,\d y\\
    &\leq C\int_{\R^d}(1+|x|^k)\,\gg(\d x)\\
    &\quad\times\int_0^t\E\left[
    \left(
    (S_t-S_r)^{(k-3)/2}+(S_t-S_r)^{-3/2}
    \right)
    \int_r^t\e^{-\delta s}[\W_\eta(\nu^1_s,\nu^2_s)+\W_k(\nu^1_s,\nu^2_s)]\cdot\e^{\delta s}\,\d S_s
    \right]\d r\\
    &\leq C\gamma(1+|\cdot|^k)\sup_{s\in[0,T]}\e^{-\delta s}[\W_\eta(\nu^1_s,\nu^2_s)+\W_k(\nu^1_s,\nu^2_s)]\\
    &\quad\times
    \int_0^t\E\left[
    \left(
    (S_t-S_r)^{(k-3)/2}+(S_t-S_r)^{-3/2}
    \right)
    \int_r^t\e^{\delta \tau}\,\d S_\tau
    \right]\d r.
\end{align*}
Combining the bounds for $|\mathsf{J}_i|$, $i=1,2,3,4$, we obtain for any $\delta>0$,
\begin{align*}
    &\sup_{t\in[0,T]}\e^{-\delta t}\big\|\L_{{X}_{t}^{\gg,\mu^1,\nu^1}}-\L_{{X}_{t}^{\gg,\mu^2,\nu^2}}\big\|_{k,var}
    \leq\sum_{i=1}^4\sup_{t\in(0,T]}\sup_{f\in\scr B_b(\R^d),|f|\leq 1+|\cdot|^k}
    \e^{-\delta t}|\mathsf{J}_i|\\
    &\leq C\gamma(1+|\cdot|^k)\sup_{s\in[0,T]}\e^{-\delta s}[\W_\eta(\nu^1_s,\nu^2_s)+\W_k(\nu^1_s,\nu^2_s)]\\
    &\quad+C\delta^{1/\alpha-1}
    \sup_{s\in[0,T]}\e^{-\delta s}\big\|\L_{{X}_{s}^{\gg,\mu^1,\nu^1}}-\L_{{X}_{s}^{\gg,\mu^2,\nu^2}}\big\|_{k,var}\\
    &\quad+C\gamma(1+|\cdot|^k)\delta^{1/\alpha-1}\sup_{s\in[0,T]}\e^{-\delta s}\left[
    \|\mu^1_s-\mu^2_s\|_{k,var}+\W_{k}(\mu^1_s,\mu^2_s)
    \right]\\
    &\quad+C\gamma(1+|\cdot|^k)\sup_{s\in[0,T]}\e^{-\delta s}[\W_\eta(\nu^1_s,\nu^2_s)+\W_k(\nu^1_s,\nu^2_s)]\\
    &\qquad\times
    \sup_{t\in(0,T]}\e^{-\delta t}\int_0^t\E\left[
    \left(
    (S_t-S_r)^{(k-3)/2}+(S_t-S_r)^{-3/2}
    \right)
    \int_r^t\e^{\delta \tau}\,\d S_\tau
    \right]\d r.
\end{align*}
Taking $\delta>0$ large enough and using Lemma \ref{slimit}\,ii) (with $\kappa=k/2$ and $\kappa=0$) below,
we derive the first assertion.

To prove the second inequality, we note that
\begin{align*}
    \W_\eta\big(\L_{{X}_{t}^{\gg,\mu^1,\nu^1}},\L_{{X}_{t}^{\gg,\mu^2,\nu^2}}\big)
    &=\sup_{f\in\scr B_b(\R^d),[f]_\eta\leq 1,f(0)=0}\left|
    \int_{\R^d}\big[
    P_{t}^{\mu^1,\nu^1}f(x)-P_{t}^{\mu^2,\nu^2}f(x)\big]\,\gg(\d x)\right|\\
    &\leq\sum_{i=1}^4\sup_{f\in\scr B_b(\R^d),[f]_\eta\leq 1,f(0)=0}|\mathsf{J}_i|\\
    &\leq\sup_{f\in\scr B_b(\R^d),[f]_\eta\leq 1}|\mathsf{J}_1|
    +\sum_{i=2}^4\sup_{f\in\scr B_b(\R^d),|f|\leq1+|\cdot|^k}|\mathsf{J}_i|.
\end{align*}
It follows from \eqref{g0'}
with $\epsilon=\eta$ that for all $t\in(0,T]$, $\delta>0$ and $f\in\scr B_b(\R^d)$ satisfying $[f]_\eta\leq 1$,
\begin{align*}
    |\mathsf{J}_1|&=\left|\int_{\R^d}\gg(\d x)
    \int_{\R^d}\big[q_{t}^{\nu^1}(x,y)-q_{t}^{\nu^2}(x,y)\big]\big(f(y)-f(x)\big)\,\d y
    \right|\\
    &\leq \int_{\R^d}\gg(\d x)\int_{\R^d}\big|q_{t}^{\nu^1}(x,y)-q_{t}^{\nu^2}(x,y)\big||y-x|^\eta\,\d y\\
    &\leq C\E\left[
    S_t^{-1+\eta/2}\int_0^t[\W_\eta(\nu^1_r,\nu^2_r)+\W_k(\nu^1_r,\nu^2_r)]\,\d S_r
    \right]\\
    &\leq C\sup_{s\in[0,t]}\e^{-\delta s}[\W_\eta(\nu^1_s,\nu^2_s)+\W_k(\nu^1_s,\nu^2_s)]
    \times\E\left[
    S_t^{-1+\eta/2}\int_0^t\e^{\delta r}\,\d S_r
    \right].
\end{align*}
This, together with the above bounds for $|\mathsf{J}_i|$, $i=2,3,4$, implies that for all $\delta>0$
\begin{align*}
     &\sup_{t\in[0,T]}\e^{-\delta t}\W_\eta\big(\L_{{X}_{t}^{\gg,\mu^1,\nu^1}},\L_{{X}_{t}^{\gg,\mu^2,\nu^2}}\big)\\
    &\qquad\leq\sup_{t\in(0,T]}\sup_{f\in\scr B_b(\R^d),[f]_\eta\leq 1}\e^{-\delta t}|\mathsf{J}_1|
    +\sum_{i=2}^4\sup_{t\in(0,T]}\sup_{f\in\scr B_b(\R^d),|f|\leq1+|\cdot|^k}\e^{-\delta t}|\mathsf{J}_i|\\
    &\qquad\leq C\sup_{s\in[0,T]}\e^{-\delta s}[\W_\eta(\nu^1_s,\nu^2_s)+\W_k(\nu^1_s,\nu^2_s)]
    \times\sup_{t\in(0,T]}\e^{-\delta t}\E\left[
    S_t^{-1+\eta/2}\int_0^t\e^{\delta r}\,\d S_r
    \right]\\
    &\qquad\quad+C\delta^{1/\alpha-1}
    \sup_{s\in[0,T]}\e^{-\delta s}\big\|\L_{{X}_{r}^{\gg,\mu^1,\nu^1}}-\L_{{X}_{r}^{\gg,\mu^2,\nu^2}}\big\|_{k,var}\\
    &\qquad\quad+C\gamma(1+|\cdot|^k)\delta^{1/\alpha-1}\sup_{s\in[0,T]}\e^{-\delta s}\left[
    \|\mu^1_s-\mu^2_s\|_{k,var}+\W_{k}(\mu^1_s,\mu^2_s)
    \right]\\
    &\qquad\quad+C\gamma(1+|\cdot|^k)\sup_{s\in[0,T]}\e^{-\delta s}[\W_\eta(\nu^1_s,\nu^2_s)+\W_k(\nu^1_s,\nu^2_s)]\\
    &\qquad\qquad\times
    \sup_{t\in(0,T]}\e^{-\delta t}\int_0^t\E\left[
    \left(
    (S_t-S_r)^{(k-3)/2}+(S_t-S_r)^{-3/2}
    \right)
    \int_r^t\e^{\delta \tau}\,\d S_\tau
    \right]\d r.
\end{align*}
Combining this with the upper bound for $\sup_{t\in[0,T]}\e^{-\delta t}\big\|\L_{{X}_{t}^{\gg,\mu^1,\nu^1}}-\L_{{X}_{t}^{\gg,\mu^2,\nu^2}}\big\|_{k,var}$ in the
first assertion and using Lemma \ref{slimit} below, we obtain the desired inequality and the proof
is now finished.
\end{proof}

\subsection{Proof of Proposition \ref{PW}}

\begin{proof}[Proof of Proposition \ref{PW}]
By Lemma \ref{PW0} and the second ineqiality in Lemma \ref{vdh22s},
we know that for all $\gamma\in\scr P_k$, $\mu^i,\nu^i\in C([0,T],\scr P_k)$, $i=1,2$, and $\delta>0$ large enough,
\begin{equation}\label{wke}
    \begin{aligned}
   &\sup_{t\in[0,T] }\e^{-\delta t}\big[\W_\eta(\L_{X_{t}^{\gg,\mu^1,\nu^1}},\L_{X_{t}^{\gg,\mu^2,\nu^2}})+ \W_k(\L_{X_{t}^{\gg,\mu^1,\nu^1}},\L_{X_{t}^{\gg,\mu^2,\nu^2}})\big]\\
   &\qquad\qquad\leq \sup_{t\in[0,T] }\e^{-\delta t}\W_\eta(\L_{X_{t}^{\gg,\mu^1,\nu^1}},\L_{X_{t}^{\gg,\mu^2,\nu^2}})
   +\sup_{t\in[0,T] }\e^{-\delta t}\W_k(\L_{X_{t}^{\gg,\mu^1,\nu^1}},\L_{X_{t}^{\gg,\mu^2,\nu^2}})\\
    &\qquad\qquad\le \frac{1}{2}\sup_{t\in[0,T]}\e^{-\delta t}[\W_\eta(\nu^1_t,\nu^2_t)+\W_k(\nu^1_t,\nu^2_t)]\\
   &\qquad\qquad\quad+C\gamma(1+|\cdot|^k)[\delta^{1/\alpha-1}+\delta^{-1}]
    \sup_{t\in[0,T]}\e^{-\delta t}\left[
    \|\mu^1_t-\mu^2_t\|_{k,var}+\W_{k}(\mu^1_t,\mu^2_t)\right].
\end{aligned}
\end{equation}
Taking $\mu^1=\mu^2=\mu\in C([0,T],\scr P_k)$, it holds that for $\delta>0$ large enough,
\begin{align*}
    \sup_{t\in[0,T] }\e^{-\delta t}\big[\W_\eta(\L_{X_{t}^{\gg,\mu,\nu^1}},&\L_{X_{t}^{\gg,\mu,\nu^2}})+ \W_k(\L_{X_{t}^{\gg,\mu,\nu^1}},\L_{X_{t}^{\gg,\mu,\nu^2}})\big]\\
    &\qquad\leq\frac{1}{2}\sup_{t\in[0,T]}\e^{-\delta t}[\W_\eta(\nu^1_t,\nu^2_t)+\W_k(\nu^1_t,\nu^2_t)].
\end{align*}
This means that, for $\delta>0$ large enough, the map
$$\nu\mapsto \L_{X_{\cdot}^{\gg,\mu,\nu}}$$
is strictly contractive in  $C([0,T];\scr P_k)$  under the complete metric
$$\sup_{t\in[0,T] }\e^{-\delta t}[\W_\eta(\nu^1_t,\nu^2_t)+\W_k(\nu^1_t,\nu^2_t)]$$
for $\nu^1,\nu^2\in C([0,T];\scr P_k)$.
Then it has a unique fixed point $\nu^\ast=\nu^\ast(\gg,\mu)\in  C([0,T];\scr P_k)$, i.e.\ $\nu=\nu^\ast$ is the
unique solution to the equation
$$
    \L_{X_{t}^{\gg,\mu,\nu}}=\nu_t,\quad t\in[0,T].
$$
Therefore, $X_{t}^{\gg,\mu}=X_{t}^{\gg,\mu,\nu^\ast}$
solves \eqref{ED}, and
this proves the first assertion.

Note that $X_{\cdot}^{\gg,\mu,\nu}=X_{\cdot}^{\gg,\mu}$ for
$\nu=\L_{X_{\cdot}^{\gg,\mu}}$. To prove the second claim, we take $\nu^1=\L_{X_{\cdot}^{\gg,\mu^1}}$ and $\nu^2=\L_{X_{\cdot}^{\gg,\mu^2}}$
in \eqref{wke} to get that for $\delta>0$ large enough,
\begin{align*}
    \sup_{t\in[0,T] }\e^{-\delta t}&\big[\W_\eta(\L_{X_{t}^{\gg,\mu^1}},\L_{X_{t}^{\gg,\mu^2}})+ \W_k(\L_{X_{t}^{\gg,\mu^1}},\L_{X_{t}^{\gg,\mu^2}})\big]\\
    &\le \frac{1}{2}\sup_{t\in[0,T]}\e^{-\delta t}\big[\W_\eta(\L_{X_{t}^{\gg,\mu^1}},\L_{X_{t}^{\gg,\mu^2}})
    +\W_k(\L_{X_{t}^{\gg,\mu^1}},\L_{X_{t}^{\gg,\mu^2}})\big]\\
    &\quad+C\gamma(1+|\cdot|^k)[\delta^{1/\alpha-1}+\delta^{-1}]
    \sup_{t\in[0,T]}\e^{-\delta t}\left[
    \|\mu^1_t-\mu^2_t\|_{k,var}+\W_{k}(\mu^1_t,\mu^2_t)\right],
\end{align*}
which immediately yields the desired estimate.
\end{proof}

\section{Proof of Theorem \ref{EUS}}

\begin{proof}[Proof of Theorem \ref{EUS}]
Taking $\nu^1=\L_{X_{\cdot}^{\gg,\mu^1}}$ and $\nu^2=\L_{X_{\cdot}^{\gg,\mu^2}}$ in the first inequality
in Lemma \ref{vdh22s}, and using the estimate in Proposition \ref{PW},
we obtain that for $\delta>0$ large enough,
\begin{align*}
    &\sup_{t\in[0,T]}\e^{-\delta t}\big[
    \|\L_{X_{t}^{\gg,\mu^1}}-\L_{X_{t}^{\gg,\mu^2}}\|_{k,var}+\W_{k}(\L_{X_{t}^{\gg,\mu^1}},\L_{X_{t}^{\gg,\mu^2}})\big]\\
    &\qquad\qquad\qquad\leq
    \sup_{t\in[0,T]}\e^{-\delta t}\|\L_{X_{t}^{\gg,\mu^1}}-\L_{X_{t}^{\gg,\mu^2}}\|_{k,var}
    +\sup_{t\in[0,T]}\e^{-\delta t}\W_{k}(\L_{X_{t}^{\gg,\mu^1}},\L_{X_{t}^{\gg,\mu^2}})\\
    &\qquad\qquad\qquad\leq C\gamma(1+|\cdot|^k)
    \sup_{t\in[0,T]}\e^{-\delta t}\big[
    \W_\eta(\L_{X_{t}^{\gg,\mu^1}},\L_{X_{t}^{\gg,\mu^2}})+
    \W_k(\L_{X_{t}^{\gg,\mu^1}},\L_{X_{t}^{\gg,\mu^2}})
    \big]\\
    &\qquad\qquad\qquad\quad+C\gamma(1+|\cdot|^k)
    \delta^{1/\alpha-1}\sup_{t\in[0,T]}\e^{-\delta t}
    \left[
    \|\mu^1_t-\mu^2_t\|_{k,var}+\W_{k}(\mu^1_t,\mu^2_t)
    \right]\\
    &\qquad\qquad\qquad\leq C\gamma(1+|\cdot|^k)^2\left[
    \delta^{1/\alpha-1}+\delta^{-1}
    \right]\sup_{t\in[0,T]}\e^{-\delta t}
    \left[
    \|\mu^1_t-\mu^2_t\|_{k,var}+\W_{k}(\mu^1_t,\mu^2_t)
    \right].
\end{align*}
We can now conclude that, for $\delta>0$ large enough, the map
$$
    \mu\mapsto \L_{X_{\cdot}^{\gg,\mu}}
$$
is strictly contractive in  $C([0,T];\scr P_k)$  under the complete metric
$$\sup_{t\in[0,T]}\e^{-\delta t}\left[
\|\mu^1_t-\mu^2_t\|_{k,var}+\W_{k}(\mu^1_t,\mu^2_t)\right]$$
for $\mu^1,\mu^2\in C([0,T];\scr P_k)$.
Then it has a unique fixed point $\mu^\ast=\mu^\ast(\gg)\in  C([0,T];\scr P_k)$
such that $\mu^\ast=\L_{X_{\cdot}^{\gg,\mu^\ast}}$, and $X_t=X_{t}^{\gg,\mu^\ast}$ is the unique solution
to \eqref{E1} with $\L_{X_0}=\gamma\in\scr P_k$.

Recall that
$$
    X_t=X_0+\int_0^tb_r(X_r,\L_{X_r})\,\d r+\int_0^t\sigma_r(\L_{X_r})\,\d W_{S_r},
$$
where $\L_{X_0}\in\scr P_k$. Since the coefficients $b$ and $\sigma$ are bounded, as in \eqref{gr4e23} it
is easy to get
$$
    \E\left[\sup_{t\in[0,T]}|X_t|^k\right]
    \leq C\E\big[|X_0|^k\big]+C.
$$
This completes the proof.
\end{proof}

\section{Appendix}\label{app}

\begin{lem}\label{slimit}
    Let $T>0$ and $S_t$ be an $\frac{\alpha}{2}$-stable subordinator \textup{(}$0<\alpha<2$\textup{)}.

  \smallskip\noindent\textup{i)}
        If $0<\kappa<\alpha/2$, then
       $$
        \lim_{\delta\rightarrow\infty}\sup_{t\in(0,T]}\e^{-\delta t}\E\left[
        S_t^{\kappa-1}\int_0^t\e^{\delta r}\,\d S_r
        \right]=0.
       $$

    \smallskip\noindent\textup{ii)}
        If $(1-\alpha)/2<\kappa<(1+\alpha)/2$, then
       $$
        \lim_{\delta\rightarrow\infty}\sup_{t\in(0,T]}\e^{-\delta t}
        \int_0^t \E\left[(S_t-S_r)^{\kappa-3/2}
        \int_r^t\e^{\delta \tau}\,\d S_\tau
        \right]\d r
        =0.
       $$
\end{lem}

\begin{proof}
    i) Since $r\mapsto S_r$ is nondecreasing, we have for any $t>0$, $\delta>0$ and $\epsilon\in(0,1)$,
    \begin{align*}
        S_t^{\kappa-1}\int_0^t\e^{\delta r}\,\d S_r
        &=S_t^{\kappa-1}\left(
        \int_0^{(1-\epsilon)t}+\int_{(1-\epsilon)t}^t
        \right)\e^{\delta r}\,\d S_r\\
        &\leq \e^{\delta(1-\epsilon)t}S_t^{\kappa-1}S_{(1-\epsilon)t}
        +\e^{\delta t}S_t^{\kappa-1}\left(S_t-S_{(1-\epsilon)t}\right)\\
        &\leq \e^{\delta(1-\epsilon)t}S_t^\kappa+\e^{\delta t}\left(S_t-S_{(1-\epsilon)t}\right)^\kappa.
    \end{align*}
    Since a subordinator has stationary increments, it follows that
    \begin{align*}
        \e^{-\delta t}\E\left[
        S_t^{\kappa-1}\int_0^t\e^{\delta r}\,\d S_r
        \right]
        &\leq\e^{-\delta\epsilon t}\E\left[S_t^\kappa\right]
        +\E\left[\left(S_t-S_{(1-\epsilon)t}\right)^\kappa\right]\\
        &=\e^{-\delta\epsilon t}\E\left[S_t^\kappa\right]
        +\E\left[S_{\epsilon t}^\kappa\right]\\
        &=\e^{-\delta\epsilon t}t^{2\kappa/\alpha}
        \E\left[S_1^\kappa\right]
        +(\epsilon t)^{2\kappa/\alpha}
        \E\left[S_1^\kappa\right],
    \end{align*}
    which implies that for any $\delta>0$ and $\epsilon\in(0,1)$,
    \begin{align*}
        \sup_{t\in(0,T]}\e^{-\delta t}\E\left[
        S_t^{\kappa-1}\int_0^t\e^{\delta r}\,\d S_r
        \right]
        &\leq\E\left[S_1^\kappa\right]\left(
        \sup_{t\in(0,T]}\e^{-\delta\epsilon t}t^{2\kappa/\alpha}
        +\sup_{t\in(0,T]}(\epsilon t)^{2\kappa/\alpha}
        \right)\\
        &\leq\E\left[S_1^\kappa\right]\left(
        \left(
        \frac{2\kappa}{\alpha\e\epsilon\delta}
        \right)^{2\kappa/\alpha}
        +(\epsilon T)^{2\kappa/\alpha}
        \right).
    \end{align*}
    By letting first $\delta\rightarrow\infty$ then $\epsilon\downarrow0$, we get the desired claim.

    \medskip

    \noindent
    ii) Pick $\theta\in\R$ such that
    $$
        1-\frac\alpha2<\theta<1\wedge\left(\frac32-\kappa\right).
    $$
    Since $r\mapsto S_r$ is nondecreasing, for any $0\leq r< t$, $\delta>0$ and $\epsilon\in(0,1)$,
    \begin{align*}
        &(S_t-S_r)^{\kappa-3/2}\int_r^t\e^{\delta \tau}\,\d S_\tau
        =(S_t-S_r)^{\kappa-3/2}\left(
        \int_r^{t-\epsilon(t-r)}+\int_{t-\epsilon(t-r)}^t
        \right)\e^{\delta \tau}\,\d S_\tau\\
        &\leq\e^{\delta\left(t-\epsilon(t-r)\right)}
        (S_t-S_r)^{\kappa-3/2}
        \left(S_{t-\epsilon(t-r)}-S_r\right)
        +\e^{\delta t}(S_t-S_r)^{\theta+\kappa-3/2}
        (S_t-S_r)^{-\theta}
        \left(
        S_t-S_{t-\epsilon(t-r)}
        \right)\\
        &\leq\e^{\delta\left(t-\epsilon(t-r)\right)}
        (S_t-S_r)^{\kappa-1/2}
        +\e^{\delta t}\left(S_{t-\epsilon(t-r)}-S_r\right)^{\theta+\kappa-3/2}
        \left(
        S_t-S_{t-\epsilon(t-r)}
        \right)^{1-\theta}.
    \end{align*}
    Because of the independent and stationary increments property
    of a subordinator, we obtain
    \begin{align*}
        &\E\left[(S_t-S_r)^{\kappa-3/2}
        \int_r^t\e^{\delta \tau}\,\d S_\tau
        \right]\\
        &\qquad\leq\e^{\delta\left(t-\epsilon(t-r)\right)}
        \E\left[
        \left(S_t-S_r\right)^{\kappa-1/2}
        \right]
        +\e^{\delta t}\E\left[
        \left(S_{t-\epsilon(t-r)}-S_r\right)^{\theta+\kappa-3/2}
        \right]
        \E\left[
        \left(
        S_t-S_{t-\epsilon(t-r)}
        \right)^{1-\theta}
        \right]\\
        &\qquad=\e^{\delta\left(t-\epsilon(t-r)\right)}
        \E\left[
        S_{t-r}^{\kappa-1/2}
        \right]
        +\e^{\delta t}\E\left[
        S_{(1-\epsilon)(t-r)}
        ^{\theta+\kappa-3/2}
        \right]
        \E\left[
        S_{\epsilon(t-r)}^{1-\theta}
        \right]\\
        &\qquad=\e^{\delta\left(t-\epsilon(t-r)\right)}
        (t-r)^{(2\kappa-1)/\alpha}
        \E\left[
        S_1^{\kappa-1/2}
        \right]\\
        &\qquad\quad+\e^{\delta t}
        \epsilon^{2(1-\theta)/\alpha}(1-\epsilon)^{(2\theta+2\kappa-3)/\alpha}
        (t-r)^{(2\kappa-1)/\alpha}
        \E\left[
        S_1^{\theta+\kappa-3/2}
        \right]
        \E\left[S_1^{1-\theta}
        \right].
    \end{align*}
    This yields that for any $\delta>0$ and $\epsilon\in(0,1)$,
    \begin{align*}
        &\sup_{t\in(0,T]}\e^{-\delta t}
        \int_0^t \E\left[(S_t-S_r)^{\kappa-3/2}
        \int_r^t\e^{\delta \tau}\,\d S_\tau
        \right]\d r\\
        &\quad\qquad\qquad\leq\E\left[
        S_1^{\kappa-1/2}
        \right]\sup_{t\in(0,T]}\int_0^t
        \e^{-\delta\epsilon(t-r)}
        (t-r)^{(2\kappa-1)/\alpha}\,\d r\\
        &\quad\qquad\qquad\quad+\epsilon^{2(1-\theta)/\alpha}(1-\epsilon)^{(2\theta+2\kappa-3)/\alpha}
        \E\left[
        S_1^{\theta+\kappa-3/2}
        \right]
        \E\left[S_1^{1-\theta}
        \right]\sup_{t\in(0,T]}\int_0^t
        (t-r)^{(2\kappa-1)/\alpha}
        \,\d r\\
        &\quad\qquad\qquad=\E\left[
        S_1^{\kappa-1/2}
        \right]\int_0^T
        \e^{-\delta\epsilon r}
        r^{(2\kappa-1)/\alpha}\,\d r\\
        &\quad\qquad\qquad\quad+\epsilon^{2(1-\theta)/\alpha}(1-\epsilon)^{(2\theta+2\kappa-3)/\alpha}
        \E\left[
        S_1^{\theta+\kappa-3/2}
        \right]
        \E\left[S_1^{1-\theta}
        \right]
        \int_0^Tr^{(2\kappa-1)/\alpha}
        \,\d r.
    \end{align*}
    It remains to let first $\delta\rightarrow\infty$ then $\epsilon\downarrow0$
    to finish the proof.
\end{proof}

\noindent
\textbf{Acknowledgement.} C.-S.\ Deng is supported by
Natural Science Foundation of Hubei Province of China (2022CFB129).
X.\ Huang is supported by National Natural Science Foundation of China (12271398).

\end{document}